\documentclass[a4paper,12pt]{article}
\usepackage[utf8]{inputenc}

\usepackage{a4wide}

\title{Complex group rings and group C\Star algebras \\ of group extensions}
\author{Johan \"{O}inert \& Stefan Wagner}

\usepackage[T1]{fontenc}
\usepackage{lmodern}
\usepackage[english]{babel}
\usepackage{hyperref}
\usepackage{amsmath, amsthm, amssymb, bbm, mathtools, nicefrac}
\usepackage{tablists,enumitem}
\usepackage{xspace}
\usepackage{xifthen}
\usepackage{enumerate}
\usepackage{fixme}
\usepackage[new]{old-arrows}
\usepackage{slashed}

\hypersetup{colorlinks=true,
  linkcolor = blue,
  citecolor = blue,
  urlcolor = blue
}

\theoremstyle{plain}
	\newtheorem{thm}{Theorem}[section]
	
	\newtheorem{lemma}[thm]{Lemma}
	\newtheorem{cor}[thm]{Corollary}

\theoremstyle{definition}
	
	\newtheorem{rmk}[thm]{Remark}
	\newtheorem{expl}[thm]{Example}
	
	\newtheorem{prob}{Problem}

\newcommand*{\Z}{\mathbb Z}		
\newcommand*{\C}{\mathbb C}		
\newcommand*{\one}{\mathbbm 1}		

\DeclareMathOperator{\id}{id}		

\DeclareMathOperator{\Supp}{Supp}		%

\newcommand{\cf}{\mbox{cf.}\xspace}			
\newcommand{\eg}{\mbox{e.\,g.}\xspace}			
\newcommand*{\ie}{\mbox{i.\,e.}\xspace}			
\newcommand{\wrt}{\mbox{w.\,r.\,t.}\xspace}		
\newcommand*{\ndash}{\nobreakdash-}

\newcommand*{\Star}{$^*$\ndash}
	
\DeclareMathOperator{\Aut}{Aut}

\DeclareMathOperator{\Out}{Out}

\DeclareMathOperator{\Ext}{Ext}

\begin{document}

\author{
Johan \"Oinert \thanks{
		Blekinge Tekniska H\"ogskola,
		\href{mailto:johan.oinert@bth.se}{\nolinkurl{johan.oinert@bth.se}}
	} \and 
	Stefan Wagner \thanks{
		Blekinge Tekniska H\"ogskola,
		\href{mailto:stefan.wagner@bth.se}{\nolinkurl{stefan.wagner@bth.se}}
	}
}
\sloppy
\maketitle

\begin{abstract}
	\noindent
	Let $N$ and $H$ be groups, and let $G$ be an extension of $H$ by $N$. 
	In this article we describe the structure of the complex group ring of $G$ in terms of data associated with $N$ and $H$. 
	In particular, we present conditions on the building blocks $N$ and $H$ guaranteeing that $G$ satisfies the 
	zero-divisor and idempotent conjectures.
	Moreover, for central extensions involving amenable groups we present conditions on the building blocks guaranteeing that the Kadison-Kaplansky conjecture holds for the group C\Star algebra of $G$.
	
	\vspace*{0,5cm}

	
	
	\noindent Data Availability Statement: This manuscript has no associated data.
\end{abstract}

\section{Introduction}

Group rings first appeared implicitly in 1854 in an article by A. Cayley~\cite{Cayley1854} and explicitly in 1897 in an article by T.~Molien~\cite{Mol97}.
They are very interesting algebraic structures whose importance became apparent after the work of \eg R.~Brauer, E.~Noether, G.~Frobenius, H.~Maschke and I.~Schur in the beginning of the last century.

With a few exceptions, the first articles on group rings of infinite groups appeared in the early 1950s. 
A key person 
in that line of research was I.~Kaplansky, known for his many deep contributions to ring theory and operator algebra.
In his famous talk, given at a conference that was held on June 6-8, 1956 at Shelter Island, Rhode Island, New York, he proposed twelve problems in the theory of rings. One of those problems was the following (see \eg~\cite{Kap56,Kap70})
which nowadays is known as the \emph{zero-divisor problem}.

\begin{prob}\label{probZ}
Let $K$ be a field and let $G$ be a torsion-free group. Is $K[G]$ a domain?
\end{prob}

Many of the problems in Kaplansky's list have been solved, and it has been shown that Problem~\ref{probZ} has an affirmative answer for several important classes of groups (see~\eg \cite{Bro1976,Cli1980,For1973,KrophollerLinnellMoody,Li77}). 
However, for a general group $G$, the answer to the problem remains unknown. 
The assertion that it has an affirmative answer is commonly referred to as \emph{Kaplansky's zero-divisor conjecture}.

Although popularized by Kaplansky, Problem~\ref{probZ} and its corresponding conjecture had in fact already been introduced by G. Higman in his 1940 thesis~\cite[p.~77]{Higman1940} (see also \cite[p.~112]{Sandling1981}).
In~\cite{Higman1940}, Higman also introduced the so-called \emph{unit problem} (see Problem~\ref{probU} below) and the corresponding \emph{unit conjecture}.

\begin{prob}\label{probU}
Let $K$ be a field and let $G$ be a torsion-free group. Is every unit in $K[G]$ trivial, \ie a scalar multiple of a group element?
\end{prob}

The unit problem 
has been answered affirmatively for special classes of groups (see \eg~\cite{CrPa2013,Pa77}), 
but in 2021 G. Gardam~\cite{Gardam} gave an example of a group ring $K[P]$ possessing non-trivial units, where $K$ is the field of two elements and $P$ is Passman's fours group. 
Building on Gardam's ideas, further counterexamples to the unit conjecture in positive characteristic were provided by A.~G.~Murray~\cite{Murray}. 
Another problem, which is closely related to the above problems, is the following.

\begin{prob}\label{probI}
Let $K$ be a field and let $G$ be a torsion-free group. Is every idempotent in $K[G]$ trivial, \ie either $0$ or $\one$?
\end{prob}

This problem is known as the \emph{idempotent problem},
and the corresponding conjecture is called \emph{the idempotent conjecture}. 
Using algebraic methods as well as analytical methods, a lot of progress (see \eg~\cite{Bur1970,For1973b} and also~\cite{Hig2001,Mar1986,Min2002,Pus2002}) has been made on Problem~\ref{probI}.
Nevertheless, for a general group $G$, the answer to Problem~\ref{probI} remains unknown. 
In the last two decades, however, Problem~\ref{probI} has regained interest, mainly due to its intimate connection with the Baum-Connes conjecture in operator algebras (see \eg~\cite{Val2002}) via the so-called Kadison-Kaplansky conjecture for reduced group C\Star algebras.

There is a mutual hierarchy between Problems~\ref{probZ},~\ref{probU} and~\ref{probI}.
Indeed, for fixed $K$ and~$G$, it is easy to see that an affirmative answer to Problem~\ref{probZ} yields that Problem~\ref{probI} has an affirmative answer.
Furthermore, using a result of D. S. Passman's (see \cite[Lem.~1.2]{Pa77}),
we conclude that an affirmative answer to Problem~\ref{probU} yields an affirmative answer to Problem~\ref{probZ}.
For a thorough account of the development on the above problems (mainly) during the 1970s, we refer the reader to \cite{Pa77}.

In this article we shall restrict our attention to complex group rings, \ie the case when $K=\C$.
Our aim is to contribute to a better understanding of Kaplansky's conjectures by studying complex group rings of group extensions. 
More concretely, let $N$ and $H$ be two groups. 
Furthermore, let $G$ be an extension of $H$ by $N$. 
Our main objective is to investigate the structure of $\C[G]$ in terms of data associated with the building blocks $N$ and $H$. 
This article is organized as follows.

In Section~\ref{sec:PrelNot} we record the most important preliminaries and notation.
In particular, we discuss crossed products and crossed systems.

In Section~\ref{sec:repcrosprod} we represent $\C[G]$ as a crossed product of the complex group ring $\C[N]$ and $H$, where the respective crossed system is associated with the factor system of the underlying group extension (see Theorem~\ref{thm:iso}). 
Although this might be well-known to experts (\cf~\cite[p.4]{NVO2004}), we have not found such a statement explicitly discussed in the literature. 
Moreover, as an application, we show that if $\C[N]$ is a domain and $H$ is a unique product group, then $G$ satisfies Kaplansky's complex
zero-divisor conjecture 
(see Theorem~\ref{thm:kap1} and Corollary~\ref{cor:kap1}). 
We conclude the section with 
several examples and remarks.

In Section~\ref{sec:repglobsec} we consider central extensions, \ie, $N$ is central in $G$, from an alternative C\Star algebraic perspective. To this end, we employ a group-adapted version of the Dauns-Hofmann Theorem (\cf~\cite{DaHo68,Ho72}) to represent the group C\Star algebra $C^*(G)$ as a C\Star algebra of global continuous sections of a C\Star algebraic bundle over the dual group $\widehat{N}$. 
In this way we are able to show, under some technical assumptions, that if $C^*(H)$ contains no non-trivial idempotent, then the same assertion holds for $C^*(G)$ (see Lemma~\ref{lem:kap2} and Corollary~\ref{cor:kap2}).

\section{Preliminaries and notation}\label{sec:PrelNot}

Our study revolves around the structure of complex group rings of group extensions. 
Consequently, we blend tools from algebraic representation theory and the theory of group extensions. 
In this preliminary section we provide the most important definitions and notation which are repeatedly used in this article. 
In general, given a group $G$, we shall always write $e_G$, or simply $1$ or $e$, for its identity element.

\subsubsection*{Group extensions and factor systems}

Let $1 \to N \to G \stackrel{q}{\to} H \to 1$ be a short exact sequence of groups. 
We first revise a description of the extension $G$ in terms of data associated with $N$ and $H$. For this purpose, let $\sigma:H\rightarrow G$ be a section of $q$, which is \emph{normalized} in the sense that $\sigma(e_H)=e_G$. 
Then the map $N\times H\rightarrow G$, $(n,h)\mapsto n\sigma(h)$ is a bijection and may be turned into an isomorphism of groups by endowing $N\times H$ with the multiplication
\begin{align}
(n,h)(n',h')=(nS(h)(n')\omega(h,h'),hh'),\label{eq:grpmulti}
\end{align}
where $S:=C_N\circ\sigma:H\rightarrow \Aut(N)$ with $C_N:G\rightarrow \Aut(N)$, $C_N(g)=gng^{-1}$ and
\begin{align*}
\omega:H\times H\rightarrow N, \quad (h,h')\mapsto \omega(h,h'):=\sigma(h)\sigma(h')\sigma(hh')^{-1}.
\end{align*}
The pair $(S,\omega)$ is called a \emph{factor system} for $N$ and $H$ and we write $N\times_{(S,\omega)} H$ for the set $N\times H$ endowed with the group multiplication defined in Equation~(\ref{eq:grpmulti}). 
We also recall that the maps $S$ and $\omega$ satisfy the relations
\begin{align}
\label{eq:actcond}
S(h)S(h')&=C_N(\omega(h,h'))S(hh'), 
\\
\label{eq:cyclecond}
\omega(h,h')\omega(hh',h'')&=S(h)(\omega(h',h''))\omega(h,h'h'')
\end{align}
for all $h,h',h''\in H$. 
For a detailed background on group extensions and factor systems we refer the reader to \cite[Chap. IV]{Mac67}.

\subsubsection*{Crossed products and crossed systems}

Let $H$ be a group and let $R=\bigoplus_{h\in H} R_h$ be a unital $H$-graded ring, \ie, $R_h R_{h'}\subseteq R_{hh'}$ for all $h,h'\in G$. 
We write $R^\times$ for the group of invertible elements of $R$ and
\begin{align*}
R^{\times}_{\text{h}}:=\bigcup_{h\in H}\left(R^\times\bigcap R_h\right)
\end{align*}
for its group of homogeneous units.
If $R^\times\bigcap R_h\neq\emptyset$ for all $h\in H$, \ie, each $R_h,h\in H$, contains an invertible element, then $R$ is called a $(R_e,H)$-\emph{crossed product}. 

Given a unital ring $R_0$ and a group $H$, a $(R_0,H)$-\emph{crossed system} is a pair $(\bar{S},\bar{\omega})$ consisting of two maps $\bar{S}:H\rightarrow\Aut(R_0)$ and $\bar{\omega}:H\times H\rightarrow R_0^\times$ satisfying the normalization conditions $\bar{S}(e)=\id_{R_0}$, $\bar{\omega}(h,e)=\bar{\omega}(e,h)=1_{R_0}$, and
\begin{align}
\label{eq:coactcond}
\bar{S}(h)\bar{S}(h')&=C_{R_0}(\bar{\omega}(h,h'))\bar{S}(hh'), 
\\
\label{eq:algcyclecond}
\bar{\omega}(h,h')\bar{\omega}(hh',h'')&=\bar{S}(h)(\bar{\omega}(h',h''))\bar{\omega}(h,h'h'')
\end{align}
for all $h,h',h''\in G$, where $C_{R_0}:R_0^\times\rightarrow\Aut(R_0)$, $C_{R_0}(r)(s):=rsr^{-1}$ denotes the canonical conjugation action. 
It is not hard to check that each crossed product gives rise to a crossed system and vice versa. For details we refer the reader to \cite{NVO2004}.

\subsubsection*{Complex group rings}

The \emph{complex group ring} $\mathbb{C}[G]$ of a group $G$ is the space of all functions $f:G\to\mathbb{C}$ with finite support endowed with the usual convolution product of functions which we shall denote by $\star$. 
Each element in $\mathbb{C}[G]$ can be uniquely written as a sum $\sum_{g\in G}f_g \delta_g$ with only finitely many non-zero coefficients $f_g \in \C$ and the Dirac functions
\begin{align*}
\delta_g:G\to\mathbb{C}, \quad \delta_g(h)=\begin{cases}
1 \quad \text{if} \,\,\, g=h\\
0 \quad \text{otherwise}.
\end{cases}
\end{align*}
Given elements $f=\sum_{g\in G}f_g \delta_g$ and $f'=\sum_{g\in G}f'_g \delta_g$ in $\mathbb{C}[G]$, we have 
\begin{align*}
 (f\star f')(h)=\sum_{g\in G}f_{g}f'_{g^{-1}h} 
\end{align*}
for all $h\in G$. 
In particular, $\mathbb{C}[G]$ is unital with identity $\delta_e$ and $G$-graded \wrt the natural decomposition $\C[G]=\bigoplus_{g\in G}\C\cdot\delta_g$. 
Since each $\delta_g$, $g\in G$, is homogeneous and invertible, $\mathbb{C}[G]$ is, in fact, a $(\C,G)$-crossed product.
Also, $\mathbb{C}[G]$ has a natural involution given by
\begin{align*}
^*:\C[G] \rightarrow \C[G], \quad f=\sum_{g\in G} f_g \delta_g \mapsto f^*:=\sum_{g\in G}\bar{f_g} \delta_{g^{-1}}
\end{align*}
and may be equipped with several appropriate norms. 
Interesting to us is the 1-norm $\Vert\cdot\Vert_1:\mathbb{C}[G]\to[0,\infty)$ defined by $\Vert f\Vert_1:=\sum_{g\in G} \vert f_g \vert$ turning $\C[G]$ into a normed \Star algebra. 
The corresponding universal enveloping C\Star algebra is the full group C\Star algebra $C^*(G)$.

\section{Representation via crossed products}\label{sec:repcrosprod}

Throughout this section, let $1 \to N \to G \stackrel{q}{\to} H \to 1$ be a short exact sequence of discrete groups. 
Furthermore, let $\sigma:H \to G$ be a section of $q$ and let $(S,\omega)$ be the corresponding factor system for $N$ and $H$. 

We wish to give a description of the complex group ring $\C[G]$ in terms of data associated with the groups $N$ and $H$. 
In fact, since $N$ is a normal subgroup of $G$, we may also consider $\C[G]$ as an $H$-graded ring 
\begin{align*}
\C[G]=\bigoplus_{h\in H}\C[N]_h
\end{align*}
with homogeneous components $\C[N]_h:=\C[N]\star\delta_{\sigma(h)}$. 
In this representation of $\C[G]$, each element can be uniquely written as a sum $\sum_{h\in H}f_h \star \delta_{\sigma(h)}$
with only finitely many non-zero coefficients $f_h \in \C[N]$. 
Furthermore, each homogeneous component contains an invertible element, and consequently $\C[G]$ is, in fact, a $(\C[N],H)$-\emph{crossed product}. 

We now provide a $(\C[N],H)$-\emph{crossed system} for the $(\C[N],H)$-\emph{crossed product} $\C[G]$ which is based on the factor system $(S,\omega)$. 
To this end, we first introduce the map 
\begin{align}\label{eq:algsec}
\bar{\sigma}:H\rightarrow \C[G]^\times, \quad \bar{\sigma}(h):=\delta_{\sigma(h)},
\end{align}
where $\C[G]^\times$ denotes the group of invertible elements of $\C[G]$. 
Then we define
\begin{align}\label{eq:algact}
\bar{S}:=C_{\C[G]}\circ\bar{\sigma}:H\rightarrow\Aut(\C[N]), 
\end{align}
where $C_{\C[G]}:\C[G]^\times\rightarrow\Aut(\C[G])$ denotes the canonical conjugation action, and
\begin{align}\label{eq:algcycle}
\bar{\omega}:H\times H \to \C[N], \quad \bar{\omega}(h,h'):=\bar{\sigma}(h)\star\bar{\sigma}(h')\star\bar{\sigma}(hh')^{-1}=\delta_{\omega(h,h')}.
\end{align}

\begin{lemma}\label{lem:algfacsys}
The pair $(\bar{S},\bar{\omega})$ is a $(\C[N],H)$-crossed system.  
\end{lemma}
\begin{proof}
It is easily seen that $\bar{S}(e)=\id_{\C[N]}$ and that $\bar{\omega}(g,e)=\bar{\omega}(e,g)=\delta_e$ for all $g\in G$. 
Next, we establish Equation~(\ref{eq:coactcond}). 
For this, let $h,h'\in H$ and let $f\in\C[N]$. 
Then a few moments thought show that
\begin{align*}
\bar{S}(h)\bar{S}(h')(f)
&=\delta_{\sigma(h)}\star\delta_{\sigma(h')} \star f \star \delta_{\sigma(h')^{-1}}\star\delta_{\sigma(h)^{-1}}
\\
&=\delta_{\sigma(h)\sigma(h')} \star f \star \delta_{\sigma(h')^{-1}\sigma(h)^{-1}}
=\delta_{\omega(h,h')\sigma(hh')} \star f \star \delta_{\sigma(hh')^{-1}\omega(h,h')^{-1}}
\\
&=\delta_{\omega(h,h')}\star\delta_{\sigma(hh')} \star f \star \delta_{\sigma(hh')^{-1}}\star\delta_{\omega(h,h')^{-1}}
=C_{\C[N]}(\bar{\omega}(h,h'))\bar{S}(hh')(f).
\end{align*}
To verify Equation~(\ref{eq:algcyclecond}), we choose $h,h',h''\in H$. 
Then a short computation yields 
\begin{align*}
\bar{S}(h)(\bar{\omega}(h',h''))\star\bar{\omega}(h,h'h'')&=\delta_{\sigma(h)}\star\delta_{\omega(h',h'')}\star\delta_{\sigma(h)^{-1}}\star\delta_{\omega(h,h'h'')}
\\
&=\delta_{\sigma(h)\omega(h',h'')\sigma(h)^{-1}\omega(h,h'h'')}=\delta_{S(h)(\omega(h',h''))\omega(h,h'h'')}.
\end{align*}
On the other hand, it is straightforwardly checked that
\begin{align*}
\bar{\omega}(h,h') \star \bar{\omega}(hh',h'')=\delta_{\omega(h,h')}\star\delta_{\omega(hh',h'')}=\delta_{\omega(h,h')\omega(hh',h'')}.
\end{align*}
Consequently, Equation~(\ref{eq:algcyclecond}) follows from the classical cocycle Equation~\eqref{eq:cyclecond}.
\end{proof}

Next, we write $\C[N]\times_{(\bar{S},\bar{\omega})} H$ for the vector space $\bigoplus_{h\in H}\C[N]d_h$ with basis $(d_h)_{h\in H}$ endowed with the multiplication $\bullet$ given on homogeneous elements by
\begin{align}
fd_h \bullet f'd_{h'}:=f \star \bar{S}(h)(f') \star \bar{\omega}(h,h') d_{hh'},\label{eq:algmulti}
\end{align}
where $f,f'\in\C[N]$ and $h,h'\in H$. It follows from Lemma~\ref{lem:algfacsys} that $\C[N]\times_{(\bar{S},\bar{\omega})} H$
is a well-defined associative algebra with multiplicative identity $d_e$. 
Moreover, a few moments thought show that $\C[N]\times_{(\bar{S},\bar{\omega})} H$ carries the structure of a $(\C[N],H)$-crossed product. 
A short computation involving the algebraic equations from Lemma~\ref{lem:algfacsys} now yields:

\begin{thm}\label{thm:iso}
Using the representation $\C[G]=\bigoplus_{h\in H}\C[N]_h$, the map  
\begin{align*}
\Phi: \C[G] \rightarrow \C[N]\times_{(\bar{S},\bar{\omega})} H, \quad f=\sum_{h\in H}f_h \star \delta_{\sigma(h)} \mapsto \sum_{h\in H} f_h d_h,
\end{align*}
is an isomorphism of $(\C[N],H)$-crossed products.
\end{thm}

We have just seen that the factor system $(S,\omega)$ gives rise to a $(\C[N],H)$-crossed product that is isomorphic to $\C[G]$. 
Conversely, keeping in mind that $N\subseteq \C[N]^\times$ via $n\mapsto\delta_n$ we have the following result: 

\begin{cor}
Suppose that $(\bar{S},\bar{\omega})$ is an abstract $(\C[N],H)$-crossed system. Then the corresponding $(\C[N],H)$-crossed product $\C[N]\times_{(\bar{S},\bar{\omega})} H$ is isomorphic to $\C[G]$ for some extension $G$ of $H$ by $N$ if $\bar{S}(H)(N)\subseteq N$ and $\emph{image}(\bar{\omega})\subseteq N$. 
If this holds, then the factor system $(S,\omega)$ for $G$ is defined by $S(h):=\bar{S}(h)_{\mid N}$, $h\in H$, and the corestriction of $\bar{\omega}$ to $N$.
\end{cor}

We now proceed to investigate the structure of the crossed product $\C[N]\times_{(\bar{S},\bar{\omega})} H$. 
Recall that a group $H$ is a \emph{unique product group} if for any two non-empty finite subsets $A,B\subseteq H$ there exists at least one element $h\in H$ which has a unique representation of the form $h=ab$ with $a\in A$ and $b\in B$.
In the next proof, given an element $f=\sum_{h\in H} f_h d_h\in\C[N]\times_{(\bar{S},\bar{\omega})} H$, we write $\Supp(f):=\{h\in H:f_h\neq 0\}$ for the corresponding support.

\begin{thm}\label{thm:kap1}
Suppose that $\C[N]$ is a domain and that $H$ is a unique product group. Then  $\C[N]\times_{(\bar{S},\bar{\omega})} H$ is a domain.
\end{thm}
\begin{proof}
Let $f=\sum_{h\in H} f_h d_h$ and $f'=\sum_{h'\in H} f_{h'}' d_{h'}$
be two non-zero elements in $\C[N]\times_{(\bar{S},\bar{\omega})} H$.
Seeking a contradiction, we suppose that $f \bullet f'=0$.
This means that
\begin{equation}\label{eq:MainThmCalc}
0 =
\sum_{h\in A} f_h d_h
\bullet 
\sum_{h'\in B} f_{h'}' d_{h'}
=
\sum_{\substack{h\in A,\\ h'\in B}} 
f_h \star \bar{S}(h)(f_{h'}') \star \bar{\omega}(h,h') d_{hh'},
\end{equation}
where
$A:=\Supp(f)$ and $B:=\Supp(f')$.
Using that $H$ is a unique product group, we find $h\in AB$, $a\in A$, and $b\in B$,
such that $h=ab$ but $h \notin (A\setminus \{a\})(B\setminus \{b\})$.
By combining this with Equation~\eqref{eq:MainThmCalc}, we get 
$f_a \star \bar{S}(a)(f_{b}') \star \bar{\omega}(a,b) d_{ab}=0$, or equivalently, 
$f_a \star \bar{S}(a)(f_{b}') \star \bar{\omega}(a,b)=0$.
It follows that
$f_a \star \bar{S}(a)(f_{b}')=0$,
which is a contradiction because $\C[N]$ is a domain and $f_a, f_b'$ are both non-zero.
\end{proof}

\begin{rmk}
Given a domain $R_0$, a unique product group $H$, and an abstract $(R_0,H)$-crossed system $(\bar{S},\bar{\omega})$, we point out that with little effort the arguments and the result of the previous theorem extends to the $(R_0,H)$-crossed product $R_0\times_{(\bar{S},\bar{\omega})} H$, that is, $R_0\times_{(\bar{S},\bar{\omega})} H$ is also a domain.
\end{rmk}

Combining Theorem~\ref{thm:iso} with
Theorem~\ref{thm:kap1}, we get the following result (\cf~\cite[p.~589]{Pa77}):

\begin{cor}\label{cor:kap1}
Suppose that $N$ satisfies Kaplansky's complex zero-divisor conjecture and that $H$ is a unique product group. Then $G$ satisfies Kaplansky's complex zero-divisor conjecture and the complex idempotent conjecture.
\end{cor}

We continue with a series of examples and remarks.

\begin{expl}\label{ex:hei}
The discrete \emph{Heisenberg group} $H_3$ is abstractly defined as the group generated by elements $a$ and $b$ such that the commutator $c = aba^{-1}b^{-1}$ is central. 
It can be realized as the multiplicative group of upper-triangular matrices
\begin{align*}
H_3:=\left\{\left(\begin{matrix}
1 & a & c\\
0 & 1 & b\\
0 & 0 & 1\end{matrix}\right):\,a,b,c\in\mathbb{Z}\right\}.
\end{align*}
Moreover, a short computation shows that $H_3$ is isomorphic (as a group) to the semidirect product $\mathbb{Z}^2\rtimes_S\mathbb{Z}$, where the semidirect product is defined by the group homomorphism
\begin{align*}
S:\mathbb{Z}\rightarrow\Aut(\mathbb{Z}^2), \quad S(k).(m,n):=(m,km+n).
\end{align*}
Consequently, Theorem~\ref{thm:iso} implies that the complex group ring $\C[H_3]$ is isomorphic to $\C[\mathbb{Z}^2]\times_{\bar{S}} \mathbb{Z}$. 
Since $\C[\mathbb{Z}^2]$ is a domain and $\mathbb{Z}$ is Abelian and torsion-free, and hence a unique product group, it follows from Corollary~\ref{cor:kap1} that $H_3$ satisfies Kaplansky's zero-divisor conjecture.
In particular, $\C[H_3]$ has no non-trivial idempotents.
\end{expl}

\begin{rmk}\label{rmk:SolvableUP}
There are two rather general situations in which we can conclude that $\C[G]$ is a domain:
\begin{itemize}
\item[1.] When $G$ is a torsion-free solvable group (see \cite[Theorem~1.4]{KrophollerLinnellMoody}).
\item[2.] When $G$ is a unique product group (see \eg Theorem~\ref{thm:kap1} with $N=\{e\}$).
\end{itemize}
\end{rmk}

\begin{expl}
Let $P$ be Passman's fours group \cite{CrPa2013,Gardam}, which is a non-split extension $1 \to \Z^3 \to P \to \Z/2\Z \times \Z/2\Z \to 1.$
It is easy to see that $P$ is a solvable group, and hence $\C[P]$ is a domain by Remark~\ref{rmk:SolvableUP}.
Now, let us consider the group $G:=P \times F_2$ where $F_2$ is the free group on two generators.
First of all, we notice that $G$ is not a unique product group because $P$ is not (\cf~\cite{Gardam}).
Indeed, there are non-empty finite subsets $A,B$ of $P$ witnessing that $P$ is not a unique product group.
The subsets $A \times \{e\}$ and $B \times \{e\}$ of $P \times F_2$ are witnessing that $G$ is not a unique product group.
Secondly, solvable groups are amenable, and subgroups of amenable groups are amenable. But $G$ contains a copy of the non-amenable group $F_2$ as a subgroup. Hence, $G$ cannot be solvable.

In light of Remark~\ref{rmk:SolvableUP}, we cannot immediately see that $\C[G]$ is a domain.
However, using Corollary~\ref{cor:kap1}, we are able to conclude that $\C[G]$ is a domain.
\end{expl}

\begin{rmk}
We emphasize that it is not known whether every unit in $\C[P]$ is trivial (\cf~\cite{Gardam}).
By Theorem~\ref{thm:iso}, $\C[P] \cong \C[\Z^3] \times_{(\bar{S},\bar{\omega})} (\Z/2\Z \times \Z/2\Z)$,
meaning that we may study the units of $\C[P]$ inside the crossed product on the right-hand side.
Note that every unit in $\C[\Z^3]$ is, as a matter of fact, trivial.
\end{rmk}

\begin{rmk}
The purpose of this remark is to illustrate how to obtain families of groups satisfying the conditions of Corollary~\ref{cor:kap1}. 
\begin{itemize}
\item[1.]
Let $N$ be a group satisfiying Kaplansky's complex zero-divisor conjecture and let $H$ be a unique product group.
Also, let $s:H\to\Out(N)$ be a group homomorphism, where $\Out(N)$ denotes the group of all outer automorphisms of $N$, and let $\Ext(H,N)_s$ be the set of equivalence classes of extensions of $H$ by $N$ inducing $s$.
It is a classical fact that $\Ext(H,N)_s$ is non-empty if and only if a certain cohomology class associated with $s$ vanishes in the third group cohomology $\text{H}^3_{\text{gr}}(H,Z(N))_s$, where $Z(N)$ stands for the center of $N$, and that in this case $\Ext(H,N)_s$ is parametrized by the second group cohomology $\text{H}^2_{\text{gr}}(H,Z(N))_s$. 
By Corollary~\ref{cor:kap1}, each of these extensions satisfies Kaplansky's complex zero-divisor conjecture.
\item[2.]
In the situation of Example~\ref{ex:hei}, the set $\Ext(\mathbb{Z},\mathbb{Z}^2)_s$ consists of a single element, where $s:=q\circ S$ and $q:\Aut(\mathbb{Z}^2)\to\Out(\mathbb{Z}^2)$ denotes the canonical projection map. 
Indeed, it is a well-known fact that $\text{H}^n_{\text{gr}}(\mathbb{Z},\mathbb{Z}^2)_s=0$ for all $n>1$.
\item[3.]
It is also possible to realize $H_3$ as a central group extension of $\mathbb{Z}$ by $\mathbb{Z}^2$ with respect to the group 2-cocycle $\omega:\mathbb{Z}^2\times\mathbb{Z}^2\to\mathbb{Z}$ defined by 
$\omega\left((k,k'),(l,l')\right):=k+l'$.
Moreover, the set $\Ext(\mathbb{Z}^2,\mathbb{Z})$ is parametrized by $\text{H}^2_{\text{gr}}(\mathbb{Z}^2,\mathbb{Z})\cong\mathbb{Z}$. 
Consequently, we obtain an infinite family of groups satisfying Kaplansky's complex zero-divisor conjecture.
\end{itemize}
\end{rmk}

\begin{rmk}
The map $\Phi$ from Theorem~\ref{thm:iso} may also be used to turn $\C[N]\times_{(\bar{S},\bar{\omega})} H$ into a \Star algebra. 
The purpose of this remark is to demonstrate that if $H$ is amenable, then  the full group C\Star algebra $C^*(G)$ is isomorphic to a suitable C\Star completion of $\C[N]\times_{(\bar{S},\bar{\omega})} H$.
For this, we regard $\C[G]$ as a dense subset of $C^*(G)$ and note that the map $\bar{\sigma}$ in Equation~\eqref{eq:algsec} actually takes values in the unitary group $U(C^*(G))$. 
Moreover, we get induced maps
\begin{align*}
\bar{S}:H\rightarrow\Aut(C^*(N)) \quad \text{and} \quad \bar{\omega}:H\times H\rightarrow U(C^*(N))
\end{align*} 
satisfying $\bar{S}(H)(\C[N])\subseteq\C[N]$ and $\text{image}(\omega)\subseteq\C[N]$.
A straightforward multiplication now shows that the multiplication and the involution are continuous for the $\ell^1$-norm
\begin{align*}
\left\Vert\sum_{h\in H} f_h d_h\right\Vert_1:=\sum_{h\in H} \Vert f_h\Vert_{C^*(N)},
\end{align*}
and we write $C^*(\C[N]\times_{(\bar{S},\bar{\omega})} H)$ for the corresponding enveloping C\Star algebra.
Finally, if $H$ is amenable, then~\cite[Prop.~4.2]{Ex96} implies that $C^*(G)$ is isomorphic to $C^*(\C[N]\times_{(\bar{S},\bar{\omega})} H)$, because both algebras are topologically graded by $H$ and generate isomorphic Fell bundles.
\end{rmk}

\section{Representation via global sections}\label{sec:repglobsec}

Let $N$ and $H$ be torsion-free and countable discrete groups with $N$ Abelian. 
Furthermore, let $G$ be a central extension of $H$ by $N$. 
Our aim is to analyze the structure of the group C\Star algebra $C^*(G)$ in terms of data associated with $N$ and $H$. 
For technical reasons, we additionally assume that $H$ is amenable. 
Then $G$ is amenable, and hence~\cite[Thm.~1.2]{PaRa92} implies that $C^*(G)$ is isomorphic to the C\Star algebra $\Gamma(E)$ of global continuous sections of a C\Star algebraic bundle $q:E\to \widehat{N}$, $\widehat{N}$ being the dual group of $N$ endowed with its natural topology turning it into a compact Hausdorff space. 
Moreover, its fibre $E_\varepsilon:=q^{-1}(\{\varepsilon\})$ at the trivial character $\varepsilon\in\widehat{N}$ is \Star isomorphic to $C^*(H)$. The proof of the next statement is very much inspired by the proof of \cite[Thm.~2.18]{EckRa18}.

\begin{lemma}\label{lem:kap2}
Let $N$ and $H$ be torsion-free and countable discrete groups with $N$ Abelian. 
Additionally, suppose that $H$ is amenable and let $G$ be a central extension of $H$ by $N$. 
If $C^*(H)$ contains no non-trivial projections, then the same is true for $C^*(G)$.
\end{lemma}
\begin{proof}
Let $p$ be a projection in $C^*(G)$. 
We write $s_p$ for the corresponding continuous global section in $\Gamma(E)$. 
By assumption, we either have $s_p(\varepsilon)=0\in E_\varepsilon$ or $s_p(\varepsilon)=1\in E_\varepsilon$.
For now, assume that $s_p(\varepsilon)=0\in E_\varepsilon$.
Since $N$ is torsion-free, its dual group $\widehat{N}$ is connected, and therefore the function 
\begin{align*}
f_p:\widehat{N} \to \mathbb{R}, \quad f_p(\chi):=\Vert s_p(\chi)\Vert_\chi,
\end{align*} 
where $\Vert\cdot\Vert_\chi$ denotes the C\Star norm of $E_\chi:=q^{-1}(\{\chi\})$, is a continuous $\{0,1\}$-valued function on a connected space with $f_p(\varepsilon)=0$. It follows that $f_p$ must be 0 everywhere which in turn implies that $s_p\equiv0$. That is, we have $p=0$.
Analogously, the case $s_p(\varepsilon)=1\in E_\varepsilon$ leads to $p=1$.
\end{proof}

\begin{cor}\label{cor:kap2}
Let $N$ and $H$ be torsion-free and countable discrete groups with $N$ Abelian. 
Additionally, suppose that $H$ is amenable and let $G$ be a central extension of $H$ by $N$. 
If $H$ satisfies the Kadison-Kaplansky conjecture, then so does the group $G$.
\end{cor}

\begin{rmk}
\begin{itemize}
\item[1.]
Since every idempotent in a C\Star algebra is similar to a projection, the conclusion of Lemma~\ref{lem:kap2} still holds in the more general context of idempotents.
\item[2.]
We would like to point out that the conclusion of Lemma~\ref{lem:kap2} can also be reached using more heavy machinery. 
In fact, since $G$ is amenable, the surjectivity of the assembly map in the Baum-Connes conjecture implies that $G$ satisfies the Kadison-Kaplansky conjecture (\cf \cite[Cor.~9.2]{Hig2001}).
\end{itemize}
\end{rmk}

\begin{expl}
It is also possible to realize the group $H_3$ from Example~\ref{ex:hei} as a central group extension of $\mathbb{Z}$ by $\mathbb{Z}^2$ with respect to the group 2-cocycle $\omega:\mathbb{Z}^2\times\mathbb{Z}^2\to\mathbb{Z}$ defined by $\omega\left((k,k'),(l,l')\right):=k+l'$. 
Applying Corollary~\ref{cor:kap2}, we can assert that $H_3$ satisfies the Kadison-Kaplansky conjecture.
\end{expl}

\section*{Acknowledgement}
The second named author would like to express his gratitude to \emph{Carl Tryggers Stiftelse f\"or Vetenskaplig Forskning} for supporting this research.

\end{document}